\documentclass[10pt,reqno]{amsart}

\usepackage{amsmath,amssymb,mathtools}
\usepackage{enumitem,placeins}
\usepackage[expansion=false]{microtype}
\usepackage{xcolor}
\usepackage{pgfplots}
\usepackage[colorlinks=true,linkcolor=blue!55!black,citecolor=blue!55!black,urlcolor=blue!55!black]{hyperref}
\pgfplotsset{compat=1.18}

\pgfmathdeclarefunction{twou}{2}{%
  \pgfmathparse{
    sqrt(max(0,(
      sqrt((4*(#1)^2+4-(#2)^2)^2+(4*(#1)*(#2))^2)
      +4*(#1)^2+4-(#2)^2
    )/2))
  }%
}
\pgfmathdeclarefunction{twov}{2}{%
  \pgfmathparse{
    sign(#1)*sqrt(max(0,(
      sqrt((4*(#1)^2+4-(#2)^2)^2+(4*(#1)*(#2))^2)
      -(4*(#1)^2+4-(#2)^2)
    )/2))
  }%
}

\newcommand{\R}{\mathbb{R}}
\newcommand{\C}{\mathbb{C}}
\newcommand{\Cminus}{\mathbb{C}_{-}}
\newcommand{\eps}{\varepsilon}
\newcommand{\ii}{\mathrm{i}}
\newcommand{\e}{\mathrm{e}}
\newcommand{\dd}{\,\mathrm{d}}
\newcommand{\one}{\mathbf{1}}
\newcommand{\Tr}{\operatorname{Tr}}
\newcommand{\Spec}{\operatorname{Spec}}
\newcommand{\dist}{\operatorname{dist}}
\newcommand{\diag}{\operatorname{diag}}
\newcommand{\Lip}{\operatorname{Lip}}
\newcommand{\BL}{\mathrm{BL}}
\newcommand{\norm}[1]{\lVert #1\rVert}
\newcommand{\ip}[2]{\langle #1,#2\rangle}
\newcommand{\set}[1]{\left\{#1\right\}}
\renewcommand{\Re}{\operatorname{Re}}
\renewcommand{\Im}{\operatorname{Im}}

\newtheorem{theorem}{Theorem}[section]
\newtheorem{proposition}[theorem]{Proposition}
\newtheorem{lemma}[theorem]{Lemma}

\theoremstyle{remark}
\newtheorem{remark}[theorem]{Remark}

\title[Uniform approximation by continuum multisolitons]
{Uniform-in-Time Approximation of Calogero--Moser Particles
by Continuum Multisolitons}
\def\AuthorEmail{vw156@georgetown.edu}
\def\FundingStatement{No funding was received for conducting this study.}
\def\CompetingInterestsStatement{The author declares no competing interests.}
\author{Hedong Wang}
\thanks{Hedong Wang (corresponding author), Department of Mathematics and
Statistics, Georgetown University, Washington, DC, USA\ifx\AuthorEmail\empty
\else.\space Email: \texttt{\AuthorEmail}\fi}
\thanks{Acknowledgments: The author thanks Benjamin Harrop-Griffiths for
introducing the problem and for comments on earlier versions of the manuscript.}
\date{August 2026}
\hypersetup{
  pdftitle={Uniform-in-Time Approximation of Calogero-Moser Particles by Continuum Multisolitons},
  pdfauthor={Hedong Wang},
  pdfkeywords={continuum Calogero-Moser equation; multisolitons; rational Calogero-Moser system; uniform approximation; rank-one perturbation}
}

\begin{document}

\begin{abstract}
We prove that every solution of the classical rational Calogero--Moser
system can be approximated, after subtracting a common linear drift from the
particle positions, by rational multisoliton solutions of the focusing
continuum Calogero--Moser equation, uniformly for all time. The construction
identifies the multisoliton inverse spectral matrix as a rank-one dissipative
perturbation of Moser's time-dependent Hermitian position matrix. A uniform
lower bound on the particle separation then yields quantitative control of
the multisoliton poles. Their real parts and real velocities approximate the
shifted particle positions and velocities with errors of order $\eps^2$,
while the pole heights are positive and sum exactly to $\eps$. Using the exact
Poisson-kernel representation of the multisoliton density, we obtain
convergence to the atomic particle measure in bounded-Lipschitz distance at
rate $O(\eps(1+|\log\eps|))$, uniformly in time. We also derive the large-time
distribution of pole height at fixed $\eps$, with one distinguished branch
retaining the limiting height and the remaining heights decaying quadratically
in time. Explicit two-particle formulas illustrate the approximation and show
that the real-position error estimate is sharp.
\end{abstract}

\maketitle
\enlargethispage{4pt}

\section{Introduction}

The focusing continuum Calogero--Moser equation is
\begin{equation}\label{eq:CCM}
  \ii\partial_tq+\partial_x^2q
  -2\ii q\,\partial_x\Pi_+\bigl(|q|^2\bigr)=0,
\end{equation}
where $\Pi_+$ denotes the Cauchy--Szeg\H{o} projection onto non-negative
wavenumbers.  Following Killip, Laurens, and Vi\c{s}an
\cite{KillipLaurensVisan}, we refer to~\eqref{eq:CCM} as the continuum
Calogero--Moser equation (CCM).  The same equation is called the
Calogero--Moser derivative nonlinear Schr\"odinger equation by G\'erard and
Lenzmann \cite{GerardLenzmann}.

The focusing model~\eqref{eq:CCM} was introduced by Abanov, Bettelheim, and
Wiegmann \cite{AbanovBettelheimWiegmann} as a hydrodynamic limit of the
Calogero--Sutherland particle system
\cite{CalogeroGround,Calogero,Sutherland71,Sutherland72,Sutherland75}.  A
defocusing version of~\eqref{eq:CCM} had previously appeared in the context of
internal waves in two-fluid systems \cite{Pelinovsky}.

Our interest in~\eqref{eq:CCM} arises from the work of G\'erard and Lenzmann
\cite{GerardLenzmann}.  Among several remarkable results, they derived a Lax
pair for solutions in the Hardy space
\[
  L_+^2(\R):=\Pi_+L^2(\R)
\]
and found explicit rational multisoliton solutions with poles $z_j(t)$ in the
lower half-plane $\Cminus$.  Whenever these poles are simple, equivalently
away from pole collisions, they evolve according to
\begin{equation}\label{eq:complexCM}
  \ddot z_j(t)
  =\sum_{k\ne j}\frac{8}{(z_j(t)-z_k(t))^3}.
\end{equation}
An equivalent determinantal formula for these multisolitons was obtained by
Matsuno \cite{Matsuno}.  Equation~\eqref{eq:complexCM} is a complex version of
the classical Calogero--Moser particle system
\begin{equation}\label{eq:realCM}
  \ddot x_j(t)
  =\sum_{k\ne j}\frac{8}{(x_j(t)-x_k(t))^3},
  \qquad x_j(t)\in\R.
\end{equation}
This system was proved to be completely integrable by Moser \cite{Moser}.  The
normalization in~\eqref{eq:realCM} is fixed throughout the paper.

The similarity between~\eqref{eq:complexCM} and~\eqref{eq:realCM} leads to the
central question of this article: can we rigorously approximate an arbitrary
solution of~\eqref{eq:realCM} by a rational multisoliton solution
of~\eqref{eq:CCM}?  We prove that the answer is yes, up to a Galilean
transformation of the particle system.

More precisely, for all sufficiently small $\eps>0$ we construct a rational
multisoliton $q^\eps(t,x)$ whose mass density $|q^\eps(t,x)|^2$ converges in the
sense of distributions to a sum of delta functions supported at the classical
particle locations.  On bounded time intervals, standard ODE techniques show
that solutions of~\eqref{eq:complexCM} with small imaginary parts are
well-approximated by solutions of~\eqref{eq:realCM}.  Here we establish the
stronger conclusion that the approximation is uniform in time.

To state the result, introduce the Hardy--Sobolev space
\[
  H_+^1(\R):=H^1(\R)\cap L_+^2(\R)
\]
and, for finite Borel measures on $\R$, the bounded-Lipschitz distance
\begin{equation}\label{eq:dBL}
  d_{\BL}(\mu,\nu)
  :=\sup\set{
    \int_{\R}\varphi\,\dd(\mu-\nu):
    \norm{\varphi}_{L^\infty}\le1,\ 
    \Lip(\varphi)\le1
  }.
\end{equation}
This distance metrizes weak convergence of finite Borel measures on $\R$.

\begin{theorem}[Uniform approximation]\label{thm:intro}
Let $N\ge1$ and let
$x=(x_1,\ldots,x_N)\in C^\infty(\R;\R^N)$ be a solution of the classical
Calogero--Moser system~\eqref{eq:realCM}.  Then there exist $c\in\R$ and
$\eps_0>0$, together with rational multisoliton solutions
\[
  q^\eps\in C(\R;H_+^1(\R)),
  \qquad 0<\eps<\eps_0,
\]
of~\eqref{eq:CCM}, such that
\begin{equation}\label{eq:introconvergence}
  \lim_{\eps\downarrow0}\ 
  \sup_{t\in\R}
  d_{\BL}\left(
    |q^\eps(t,x)|^2\dd x,\,
    2\pi\sum_{j=1}^N\delta_{x_j(t)-ct}
  \right)
  =0.
\end{equation}
In particular,
\[
  |q^\eps(t,x)|^2
  \longrightarrow
  2\pi\sum_{j=1}^N\delta_{x_j(t)-ct}
\]
in the sense of distributions, uniformly for $t\in\R$.
\end{theorem}

\begin{remark}[Quantitative form]\label{rem:quantitative-intro}
The proof gives the more precise estimate
\[
  \sup_{t\in\R}
  d_{\BL}\left(
    |q^\eps(t,x)|^2\dd x,\,
    2\pi\sum_{j=1}^N\delta_{x_j(t)-ct}
  \right)
  \le C\eps(1+|\log\eps|),
\]
where $C$ is independent of $t$ and $\eps$, but may depend on $N$, $x(0)$,
and $\dot x(0)$.
\end{remark}

The translation by $ct$ corresponds to a Galilean transformation of the
particle system~\eqref{eq:realCM}.  Na\"ively, such a translation is unnecessary
because~\eqref{eq:CCM} also possesses the Galilean symmetry
\[
  q(t,x)\longmapsto
  \e^{\ii(cx/2-c^2t/4)}q(t,x-ct).
\]
However, this transformation translates the Fourier support of $q$ by $c/2$.
Consequently, it preserves $L_+^2(\R)$ when $c\ge0$, but not when $c<0$.
Although it is possible to study~\eqref{eq:CCM} outside $L_+^2(\R)$
\cite{ChapoutoForlanoLaurens,ChapoutoForlanoLaurensSubcritical,Hadama},
we retain the Galilean shift in order to clarify the relationship between our
approximation result and the multisoliton solutions found in
\cite{GerardLenzmann}.

The key observation behind Theorem~\ref{thm:intro} is a finite-dimensional
comparison.  Moser's solution formula writes the real particles as the
eigenvalues of an affine Hermitian matrix $H(t)$.  After choosing $c$ to be one
of the asymptotic particle velocities, the G\'erard--Lenzmann inverse spectral
formula for a corresponding multisoliton produces the matrix
\begin{equation}\label{eq:rankoneintro}
  \mathcal M_\eps(t)
  =H(t)-ctI-\ii\eps P,
\end{equation}
where $P$ is an orthogonal projection of rank one.  The Hermitian matrix
$H(t)-ctI$ has a minimum gap between its eigenvalues bounded below uniformly
in time.  The perturbation in~\eqref{eq:rankoneintro} has norm $\eps$,
independently of $t$, so its eigenvalues remain uniformly close to the real
particles.  A scalar characteristic-polynomial identity will give a sharper
$O(\eps^2)$ error for their real parts; see
Theorem~\ref{thm:quantitative} below.

To the best of our knowledge, identifying the Hermitian part of
$\mathcal M_\eps(t)$ with the shifted time-dependent position matrix
$H(t)-ctI$ of an arbitrary solution of~\eqref{eq:realCM}, and using this
comparison to approximate particle trajectories by multisolitons, have not
previously been considered.  This rank-one comparison is the key to the
uniform estimates proved below.

Recent work on the continuum model includes well-posedness results in and
outside the Hardy space
\cite{BadreddineGlobal,ChapoutoForlanoLaurens,
ChapoutoForlanoLaurensSubcritical,Hadama,KillipLaurensVisan}.
Its Hamiltonian formulation and
commuting flows were developed by Killip, Marsden, and Vi\c{s}an
\cite{KillipMarsdenVisan}.  Long-time growth, zero-dispersion behavior, and
soliton resolution are considered in
\cite{HoganKowalski,BadreddineZeroDisp,KimKwon}, respectively, while direct
scattering is studied in \cite{FrankRead,SunWangZhao}.  As far as we are aware,
the problem of approximating particles by multisolitons considered in this
article has not previously been addressed.

Section~\ref{sec:prelim} reviews the classical Moser matrix and the
multisoliton inverse formula.  Section~\ref{sec:proof} proves the pole-tracking
estimates and establishes density concentration.  Section~\ref{sec:large-time}
derives the large-time pole-height asymptotics, and Section~\ref{sec:example}
gives the two-particle formulas explicitly.

\section{Particles and multisolitons}\label{sec:prelim}

\subsection{The classical particle system}

The Hamiltonian of~\eqref{eq:realCM} is
\begin{equation}\label{eq:hamiltonian}
  \mathcal E(x,\dot x)
  :=\frac12\sum_{j=1}^N\dot x_j^2
  +4\sum_{1\le j<k\le N}\frac{1}{(x_j-x_k)^2}.
\end{equation}
It is constant along every solution.  In particular, when $N\ge2$,
$\mathcal E(x,\dot x)>0$ and
\begin{equation}\label{eq:energygap}
  |x_j(t)-x_k(t)|
  \ge \frac{2}{\sqrt{\mathcal E(x,\dot x)}},
  \qquad j\ne k,\quad t\in\R.
\end{equation}
Consequently, the minimum gap between particles,
\begin{equation}\label{eq:dstar}
  d_*:=\inf_{t\in\R}\min_{j\ne k}|x_j(t)-x_k(t)|,
\end{equation}
satisfies $d_*\ge2/\sqrt{\mathcal E(x,\dot x)}>0$.

Following the presentation in \cite{Ruijsenaars}, we recall the matrix form of
Moser's integration.  After a fixed relabelling, we may assume that
$x_1(0)<\cdots<x_N(0)$.  Set $v_j=\dot x_j(0)$ and form the diagonal matrix
\[
  X_0=\diag(x_1(0),\ldots,x_N(0)).
\]
Define also the Hermitian matrix $L_0$ by
\begin{equation}\label{eq:L0}
  (L_0)_{jk}
  =v_j\delta_{jk}
  -\frac{2\ii(1-\delta_{jk})}{x_j(0)-x_k(0)}.
\end{equation}
Using the integrability of~\eqref{eq:realCM}, Moser proved that the particle
positions at time $t\in\R$ are the eigenvalues of the matrix $X_0+tL_0$
\cite{Moser}.  A direct calculation gives the commutator identity
\[
  [X_0,L_0]=-2\ii(\one\one^*-I),
\]
where $\one=(1,\ldots,1)^T\in\C^N$.

The eigenvalues $c_1,\ldots,c_N$ of $L_0$ are distinct.  Indeed, if $F$ is the
spectral projection of $L_0$ associated with an eigenvalue $c$, then
$F[X_0,L_0]F=0$.  Hence the commutator identity gives
\[
  F=(F\one)(F\one)^*,
\]
so $F$ has rank one.

Diagonalize $U^*L_0U=\diag(c_1,\ldots,c_N)$ and put $A=U^*X_0U$.  The
commutator identity gives us that
$[A,\diag(c)]=-2\ii(ww^*-I)$, where $w=U^*\one$.  Its diagonal entries give
$|w_j|=1$, so, after possibly conjugating by a diagonal unitary matrix, we may
assume that $w=\one$.  The off-diagonal entries then give
\[
  A_{jk}=\frac{2\ii}{c_j-c_k}\qquad(j\ne k).
\]
The diagonal entries of $A$ are real; write them as $\gamma_j$.  Hence the
ordered components of the solution are the eigenvalues of
\begin{equation}\label{eq:H}
  H(t)_{jk}
  =(\gamma_j+c_jt)\delta_{jk}
  +\frac{2\ii(1-\delta_{jk})}{c_j-c_k}.
\end{equation}
Conversely, the ordered eigenvalues of~\eqref{eq:H} give a global solution
of~\eqref{eq:realCM}.  The $c_j$ are its asymptotic velocities.  Complex
conjugation reverses the sign of the imaginary off-diagonal entries without
changing the particle positions.

For later use, fix an index $m\in\{1,\ldots,N\}$, let $e_m$ denote the $m$th
standard coordinate vector in $\C^N$, and write
\begin{equation}\label{eq:P}
  P_m=e_me_m^*.
\end{equation}
The following elementary property ensures that the dissipative perturbation
in~\eqref{eq:rankoneintro} moves every eigenvalue strictly into the lower
half-plane.

\begin{lemma}[Nonzero coordinates]\label{lem:nonzero}
Every eigenvector of the matrix $H(t)$ in~\eqref{eq:H} has all of its
coordinates nonzero.
\end{lemma}

\begin{proof}
Fix $t$, let $u\in\C^N$ be an eigenvector of $H$ with eigenvalue $\xi\in\R$,
and set
\[
  S_k:=\sum_{j\ne k}\frac{u_j}{c_k-c_j}.
\]
The $k$th coordinate equation is
\[
  (\gamma_k+c_kt-\xi)u_k+2\ii S_k=0.
\]
After multiplication by $\overline{u_k}$, its imaginary part gives
$\Re(\overline{u_k}S_k)=0$.

Introduce the rational functions
\[
  R(z)=\sum_{j=1}^N\frac{u_j}{z-c_j},
  \qquad
  R^\sharp(z)=\sum_{j=1}^N\frac{\overline{u_j}}{z-c_j}.
\]
The product $R(z)R^\sharp(z)$ is a rational function with poles at the $c_k$.
The principal part of its Laurent expansion at $z=c_k$ is
\[
  \frac{|u_k|^2}{(z-c_k)^2}
  +\frac{2\Re(\overline{u_k}S_k)}{z-c_k}
  =\frac{|u_k|^2}{(z-c_k)^2}.
\]
Since $R(z)R^\sharp(z)\to0$ as $|z|\to\infty$, the Mittag--Leffler theorem
yields
\begin{equation}\label{eq:rationalidentity}
  R(z)R^\sharp(z)
  =\sum_{j=1}^N\frac{|u_j|^2}{(z-c_j)^2}.
\end{equation}
If $u_m=0$, then the $m$th coordinate equation says $R(c_m)=0$.
Evaluating~\eqref{eq:rationalidentity} at $c_m$ gives
\[
  0=R(c_m)R^\sharp(c_m)
  =\sum_{j\ne m}\frac{|u_j|^2}{(c_m-c_j)^2},
\]
which requires $u=0$, a contradiction.
\end{proof}

\subsection{The continuum multisoliton formula}

We use the following part of the G\'erard--Lenzmann inverse-spectral
construction
\cite[Proposition~6.2, Theorem~6.1, and Remark~6.1]{GerardLenzmann}.  Let
$\lambda_1,\ldots,\lambda_N$ be distinct real numbers, choose an index $m$
with $\lambda_m=0$, let $\rho>0$, and choose
$\gamma_1,\ldots,\gamma_N\in\R$.  The surjectivity of the spectral mapping in
\cite[Remark~6.1]{GerardLenzmann} allows these $\gamma_j$ to be prescribed
arbitrarily.  Define
\begin{equation}\label{eq:GLmatrix}
  \mathcal M(t)
  =2\diag(\lambda_1,\ldots,\lambda_N)t+W,
\end{equation}
where
\begin{equation}\label{eq:GLW}
  W_{jk}=
  \begin{cases}
    \gamma_j-\ii\rho\,\delta_{jm},&j=k,\\[0.4ex]
    \displaystyle\frac{\ii}{\lambda_j-\lambda_k},&j\ne k.
  \end{cases}
\end{equation}
Then, for a constant phase $\phi$, the formula
\begin{equation}\label{eq:GLinverse}
  q(t,x)
  =\sqrt{2\rho}\,\e^{\ii\phi}
  \ip{(\mathcal M(t)-xI)^{-1}\one}{e_m}
\end{equation}
defines a rational multisoliton in $C(\R;H_+^1(\R))$.  Its poles are the
eigenvalues of $\mathcal M(t)$.  When the poles
$z_j(t)=\alpha_j(t)-\ii\beta_j(t)$ are simple, with
$\beta_j(t)>0$, the inverse formula also gives the exact identity
\begin{equation}\label{eq:poisson}
  |q(t,x)|^2
  =\sum_{j=1}^N
  \frac{2\beta_j(t)}
  {(x-\alpha_j(t))^2+\beta_j(t)^2}.
\end{equation}

To connect~\eqref{eq:GLmatrix} to the real particles, choose any index $m$,
take the $\gamma_j$ to be the specific values appearing in~\eqref{eq:H}, and
set
\begin{equation}\label{eq:parameters}
  \lambda_j=\frac{c_j-c_m}{2},
  \qquad
  \rho=\eps.
\end{equation}
With this choice of parameters, the matrix~\eqref{eq:GLmatrix} becomes
\begin{equation}\label{eq:keymatrix}
  \mathcal M_\eps(t)
  =H(t)-c_mtI-\ii\eps P_m.
\end{equation}
Thus the poles of the multisoliton~\eqref{eq:GLinverse} are a rank-one
perturbation of the eigenvalues
\begin{equation}\label{eq:y}
  y_j(t):=x_j(t)-c_mt.
\end{equation}
This is the particle-level translation used in Theorem~\ref{thm:intro}.

\section{Approximating particles by multisolitons}\label{sec:proof}

Our goal in this section is to prove a quantitative version of
Theorem~\ref{thm:intro}.  Recalling the definition of $d_*$
from~\eqref{eq:dstar}, we have the following.

\begin{theorem}[Quantitative uniform approximation]\label{thm:quantitative}
Let $N\ge2$, let $H(t)$ be given by~\eqref{eq:H}, and fix
$m\in\{1,\ldots,N\}$.  For all sufficiently small $\eps>0$, the matrix
$\mathcal M_\eps(t)$ in~\eqref{eq:keymatrix} has $N$ simple eigenvalues in
$\Cminus$.  They may be labelled $z_j^\eps(t)$ so that
\begin{equation}\label{eq:complexerror}
  \sup_{t\in\R}\max_j
  |z_j^\eps(t)-y_j(t)|
  \le\eps.
\end{equation}
Here $y_j(t)$ is defined as in~\eqref{eq:y}.

Let $E_j(t)$ be the spectral projection of $H(t)$ associated with the
eigenvalue $x_j(t)$, and define
\begin{equation}\label{eq:a}
  a_j(t):=\Tr(P_mE_j(t)).
\end{equation}
Then $a_j(t)>0$, $\sum_ja_j(t)=1$, and
\begin{equation}\label{eq:expansion}
  z_j^\eps(t)
  =y_j(t)-\ii\eps a_j(t)
  -\eps^2a_j(t)
    \sum_{k\ne j}\frac{a_k(t)}{y_j(t)-y_k(t)}
  +r_j^\eps(t),
\end{equation}
where
\begin{equation}\label{eq:remainder}
  \sup_{t\in\R}\max_j
  \left(
    |r_j^\eps(t)|+|\partial_t r_j^\eps(t)|
  \right)
  \le C\eps^3.
\end{equation}
Consequently,
\begin{align}
  \sup_{t\in\R}\max_j
  |\Re z_j^\eps(t)-y_j(t)|
  &\le C\eps^2,\label{eq:realerror}\\
  \sup_{t\in\R}\max_j
  |\Re\dot z_j^\eps(t)-\dot y_j(t)|
  &\le C\eps^2.\label{eq:velocityerror}
\end{align}
If $\beta_j^\eps(t):=-\Im z_j^\eps(t)$, then
\begin{equation}\label{eq:heights}
  0<\beta_j^\eps(t)<\eps,
  \qquad
  \sum_{j=1}^N\beta_j^\eps(t)=\eps,
  \qquad
  \beta_j^\eps(t)=\eps a_j(t)+O(\eps^3),
\end{equation}
uniformly in time.

Let $q^\eps$ be the multisoliton defined by~\eqref{eq:GLinverse} and
\eqref{eq:parameters}.  Then
\begin{equation}\label{eq:densityestimate}
  \sup_{t\in\R}
  d_{\BL}\left(
    |q^\eps(t,x)|^2\dd x,\,
    2\pi\sum_{j=1}^N\delta_{y_j(t)}
  \right)
  \le C\eps(1+|\log\eps|).
\end{equation}
All constants $C$ above depend only on $N$, $x(0)$, and $\dot x(0)$; in
particular, they are uniform in $t$ and $\eps$.
\end{theorem}

\begin{remark}
For $N=1$, one has
$z^\eps(t)=x(t)-c_1t-\ii\eps$, and all conclusions follow directly.
\end{remark}

\subsection{Localization and the lower half-plane}

Our starting point for the proof of Theorem~\ref{thm:quantitative} will be to
understand how the eigenvalues of $\mathcal M_\eps(t)$ and $H(t)$ are related.

Let $K=K^*$ be an $N\times N$ complex matrix with simple eigenvalues
$y_1<\cdots<y_N$ and spectral projections $E_1,\ldots,E_N$, and let $P$ be an
orthogonal projection of rank one.  We make the assumption that
\begin{equation}\label{eq:nonzerooverlap}
  \Tr(PE_j)>0
\end{equation}
for every $j$, and write $a_k:=\Tr(PE_k)$.  Away from $\Spec K$, the matrix
determinant lemma
\[
  \det(G+uv^*)=\det(G)\bigl(1+v^*G^{-1}u\bigr)
\]
gives
\begin{equation}\label{eq:detlemma}
  \det(zI-K+\ii\eps P)
  =\det(zI-K)
    \left(
      1+\ii\eps\sum_{k=1}^N\frac{a_k}{z-y_k}
    \right).
\end{equation}

\begin{lemma}[Rank-one localization]\label{lem:localization}
Let $K=K^*$ have simple eigenvalues
$y_1<\cdots<y_N$ with minimal gap at least $d>0$, and let $P$ be an
orthogonal projection of rank one satisfying~\eqref{eq:nonzerooverlap}.  Then,
if $\eps>0$ is sufficiently small depending on $d$, each disk
$\overline{D(y_j,\eps)}$ contains exactly one eigenvalue of
$K-\ii\eps P$, counted according to algebraic multiplicity.  These
eigenvalues are simple and lie strictly in $\Cminus$.
\end{lemma}

\begin{proof}
If $\dist(z,\Spec K)>\eps$, then
\[
  \norm{(z-K)^{-1}(\ii\eps P)}
  \le\frac{\eps}{\dist(z,\Spec K)}<1.
\]
Hence $z-(K-\ii\eps P)$ is invertible by a Neumann series.  The spectrum is
therefore contained in the union of the closed $\eps$-disks centered at the
$y_j$.

For the root count, \eqref{eq:nonzerooverlap} gives $a_k>0$, while
$\sum_ka_k=\Tr(P)=1$.  On the circle $|z-y_j|=d/3$,
\[
  \eps\left|
    \sum_{k=1}^N\frac{a_k}{z-y_k}
  \right|
  \le\eps\sum_{k=1}^N\frac{3a_k}{d}
  \le\frac{3\eps}{d}.
\]
This is less than one when $\eps<d/3$.  Rouch\'e's theorem, applied to the two
characteristic polynomials in~\eqref{eq:detlemma}, shows that the perturbed
polynomial has exactly one root in each disk $D(y_j,d/3)$.  Our choice of
$\eps$ ensures that these disks are disjoint and that each contains exactly
one of the disks $D(y_k,\eps)$.  Together with the preceding spectral
containment, this shows that each $\eps$-disk contains exactly one eigenvalue.
Since that eigenvalue has algebraic multiplicity one, it is simple.

Finally, if $(K-\ii\eps P)v=zv$, then
\begin{equation}\label{eq:dissipative}
  \Im z\,\norm v^2=-\eps\norm{Pv}^2\le0.
\end{equation}
If equality held, then $Pv=0$ and $Kv=zv$, contradicting
\eqref{eq:nonzerooverlap}.  Thus $\Im z<0$.
\end{proof}

Let us now apply Lemma~\ref{lem:localization} with
$K(t)=H(t)-c_mtI$ and $P=P_m$.  The minimum gap between its eigenvalues is
bounded below by $d_*$, uniformly in $t$.  Moreover, if $v$ is a unit
eigenvector of $H(t)$ and $E=vv^*$ is its spectral projection, then
Lemma~\ref{lem:nonzero} gives
\[
  \Tr(P_mE)=\ip{P_mv}{v}=|v_m|^2>0,
\]
which verifies~\eqref{eq:nonzerooverlap}.  This proves the simplicity,
lower-half-plane location, and~\eqref{eq:complexerror} in
Theorem~\ref{thm:quantitative}.

\subsection{The scalar displacement equation}

Fix $t$ temporarily and suppress it from the notation.  Let
$z_j^\eps$ be the eigenvalue near $y_j$ and write
\begin{equation}\label{eq:zeta}
  z_j^\eps=y_j+\zeta_j,
  \qquad \zeta_j\in\Cminus.
\end{equation}
Here $\zeta_j$ is a complex displacement; it need not be purely imaginary.

Define
\begin{equation}\label{eq:S}
  S_j(\zeta)
  :=\sum_{k\ne j}\frac{a_k}{y_j+\zeta-y_k}.
\end{equation}

\begin{lemma}[Displacement equation]\label{lem:scalar}
The displacement in~\eqref{eq:zeta} satisfies
\begin{equation}\label{eq:scalarequation}
  \zeta_j\bigl(1+\ii\eps S_j(\zeta_j)\bigr)
  =-\ii\eps a_j.
\end{equation}
Moreover,
\begin{equation}\label{eq:firstbound}
  |\zeta_j+\ii\eps a_j|
  \le\frac{\eps^2}{d_*-\eps},
\end{equation}
and
\begin{equation}\label{eq:secondorder}
  \zeta_j
  =-\ii\eps a_j
  -\eps^2a_j\sum_{k\ne j}\frac{a_k}{y_j-y_k}
  +O(\eps^3),
\end{equation}
uniformly for $t\in\R$.
\end{lemma}

\begin{proof}
Take $K(t)=H(t)-c_mtI$ and $P=P_m$ as above.  Since
$\zeta_j\in\Cminus$, the identity~\eqref{eq:scalarequation} follows from
applying~\eqref{eq:detlemma} with $z=y_j+\zeta_j$ and multiplying by
$\zeta_j$.

By~\eqref{eq:complexerror}, $|\zeta_j|\le\eps$.  Again using
$\sum_ka_k=1$, we have
\[
  |S_j(\zeta_j)|
  \le\frac{1}{d_*-\eps}.
\]
Consequently,~\eqref{eq:scalarequation} gives
\[
  |\zeta_j+\ii\eps a_j|
  =\eps|\zeta_j|\,|S_j(\zeta_j)|
  \le\frac{\eps^2}{d_*-\eps},
\]
which is~\eqref{eq:firstbound}.

For $|\zeta|\le\eps$,
\[
  |S_j'(\zeta)|
  \le\frac{1}{(d_*-\eps)^2}.
\]
It follows that
\[
  S_j(\zeta_j)=S_j(0)+O(\eps),
  \qquad
  S_j(0)=\sum_{k\ne j}\frac{a_k}{y_j-y_k},
\]
uniformly in time.  The exact decomposition
\begin{align*}
  \zeta_j
  &=-\ii\eps a_j-\ii\eps\zeta_jS_j(\zeta_j)\\
  &=-\ii\eps a_j-\eps^2a_jS_j(0)
    -\eps^2a_j\bigl[S_j(\zeta_j)-S_j(0)\bigr]\\
  &\hspace{3.5em}
    -\ii\eps(\zeta_j+\ii\eps a_j)S_j(\zeta_j)
\end{align*}
and~\eqref{eq:firstbound} now prove~\eqref{eq:secondorder}.
\end{proof}

\subsection{Uniform control after one time derivative}

Lemma~\ref{lem:scalar} describes the position of the poles of our
multisoliton.  In this subsection we consider the velocities, making the
corresponding expansion uniform in $C^1(\R_t)$.  The ordered eigenvalues and
spectral projections of $H(t)$ are $C^1$ because $H$ is $C^1$ and the minimum
gap between its eigenvalues is bounded below by $d_*$.  Indeed, we may readily
compute that
\begin{equation}\label{eq:eigenderivatives}
  \dot x_j=\Tr(E_j\dot H),
  \qquad
  \dot E_j
  =\sum_{k\ne j}
  \frac{E_k\dot H E_j+E_j\dot H E_k}{x_j-x_k}.
\end{equation}
These formulas follow either by differentiating the eigenvalue equations or
from the Riesz projection formula.  Since
$\dot H=\diag(c_1,\ldots,c_N)$, they imply uniform bounds for
$\dot x_j$, $\dot E_j$, and
\begin{equation}\label{eq:adot}
  \dot a_j=\Tr(P_m\dot E_j).
\end{equation}
The $c_j$ are the eigenvalues of $L_0$, so their uniform bound, and hence all
constants used here, depend only on $N$, $x(0)$, and $\dot x(0)$.

At fixed $\zeta$, differentiation of~\eqref{eq:S} gives
\begin{equation}\label{eq:Sderivatives}
  \partial_\zeta S_j(\zeta)
  =-\sum_{k\ne j}
    \frac{a_k}{(y_j+\zeta-y_k)^2},
\end{equation}
and
\begin{equation}\label{eq:Stime}
  \partial_t S_j(\zeta)
  =\sum_{k\ne j}
  \left[
    \frac{\dot a_k}{y_j+\zeta-y_k}
    -\frac{a_k(\dot y_j-\dot y_k)}
    {(y_j+\zeta-y_k)^2}
  \right].
\end{equation}
All quantities in~\eqref{eq:Sderivatives}--\eqref{eq:Stime} are uniformly
bounded when $|\zeta|\le\eps$ and $\eps<d_*/2$.
Differentiating the right-hand side of~\eqref{eq:Stime} with respect to
$\zeta$, and using the minimum gap together with the uniform bounds for
$a_k$, $\dot a_k$, and $\dot y_k$, gives
\[
  |\partial_\zeta\partial_tS_j(\zeta)|\le C.
\]
Consequently,
\[
  (\partial_tS_j)(\zeta_j)=(\partial_tS_j)(0)+O(\eps)
\]
uniformly in time.

Differentiating~\eqref{eq:scalarequation} now yields
\begin{equation}\label{eq:zetadot}
  \left(
    1+\ii\eps S_j
    +\ii\eps\zeta_j\partial_\zeta S_j
  \right)\dot\zeta_j
  =
  -\ii\eps\dot a_j
  -\ii\eps\zeta_j\partial_tS_j.
\end{equation}
The coefficient of $\dot\zeta_j$ is uniformly bounded away from zero for
small $\eps$.  In particular, \eqref{eq:zetadot} gives
\[
  \dot\zeta_j=-\ii\eps\dot a_j+O(\eps^2).
\]
Using this together with
$\zeta_j=-\ii\eps a_j+O(\eps^2)$ in
\eqref{eq:scalarequation} and~\eqref{eq:zetadot} gives the explicit
second-order formulas
\begin{align*}
  \zeta_j
  &=-\ii\eps a_j-\eps^2a_jS_j(0)+O(\eps^3),\\
  \dot\zeta_j
  &=-\ii\eps\dot a_j
    -\eps^2\bigl(\dot a_jS_j(0)+a_j\dot S_j(0)\bigr)
    +O(\eps^3),
\end{align*}
where
\[
  \dot S_j(0)
  =\sum_{k\ne j}
  \left[
    \frac{\dot a_k}{y_j-y_k}
    -\frac{a_k(\dot y_j-\dot y_k)}{(y_j-y_k)^2}
  \right].
\]
The bounds above make the remainders uniform in time.  Hence
\begin{equation}\label{eq:C1remainder}
  \sup_{t\in\R}\max_j
  \left|
    \partial_t^\ell
    \left[
      \zeta_j+\ii\eps a_j
      +\eps^2a_jS_j(0)
    \right]
  \right|
  \le C\eps^3,
  \qquad \ell=0,1.
\end{equation}
This proves~\eqref{eq:expansion}--\eqref{eq:velocityerror}.

Taking imaginary parts in~\eqref{eq:expansion} gives
$\beta_j^\eps=\eps a_j+O(\eps^3)$.  Strict positivity follows from
Lemma~\ref{lem:localization}, while
\[
  \sum_{j=1}^N\beta_j^\eps(t)
  =-\Im\Tr\mathcal M_\eps(t)
  =\eps\Tr(P_m)
  =\eps.
\]
The sum and strict positivity also imply $\beta_j^\eps<\eps$.  This proves
\eqref{eq:heights}.

\subsection{Density concentration}\label{sec:density}

We pass from pole locations to the multisoliton density.
For $0<h\le1$ and $a,b\in\R$, the Poisson kernel satisfies
\begin{equation}\label{eq:onepoisson}
  d_{\BL}\left(
    \frac{2h}{(x-a)^2+h^2}\dd x,\,
    2\pi\delta_b
  \right)
  \le
  Ch(1+|\log h|)+2\pi|a-b|.
\end{equation}
Indeed, for a test function in~\eqref{eq:dBL}, subtract $\varphi(a)$ and
split the integral into $|x-a|\le1$ and $|x-a|>1$.  The first part is bounded
using $|\varphi(x)-\varphi(a)|\le|x-a|$ and contributes
$Ch|\log h|$; the second uses
$|\varphi(x)-\varphi(a)|\le2$ and contributes $Ch$.  Replacing
$\varphi(a)$ by $\varphi(b)$ costs at most $2\pi|a-b|$.

Apply~\eqref{eq:onepoisson} to each term in~\eqref{eq:poisson}, with
\[
  a=\Re z_j^\eps(t),
  \qquad
  b=y_j(t),
  \qquad
  h=\beta_j^\eps(t).
\]
Equations~\eqref{eq:realerror} and~\eqref{eq:heights}, followed by summation
over the fixed number $N$ of poles, prove~\eqref{eq:densityestimate} and
complete the proof of Theorem~\ref{thm:quantitative}.
Choosing any index $m$ in the Moser representation and setting $c=c_m$ now
proves Theorem~\ref{thm:intro}.

\section{Large-time distribution of pole height}\label{sec:large-time}

In this section we investigate the asymptotic behavior of the poles of our
multisoliton approximation.  The labels are tied to the distinct velocities
$c_k-c_m$, not to the order of the particles on the real line.  The fixed-$\eps$ pole-height
expansions for $k\ne m$ below are those of
\cite[Lemma~6.2]{GerardLenzmann}.  For completeness, we rederive the $k\ne m$
height correction in our normalization and recover the distinguished branch from
the exact trace identity $-\sum_{k=1}^N\Im z_k^\eps(t)=\eps$.

Set
\begin{equation}\label{eq:LambdaB0}
  \Lambda:=\diag(c_1-c_m,\ldots,c_N-c_m),
  \qquad
  K(t):=H(t)-c_mtI=t\Lambda+B_0,
\end{equation}
where
\begin{equation}\label{eq:B0}
  (B_0)_{jj}=\gamma_j,
  \qquad
  (B_0)_{jk}=\frac{2\ii}{c_j-c_k}\quad(j\ne k).
\end{equation}
For $t>0$, put $s=t^{-1}$.  The diagonal entries of $\Lambda$ are distinct, so
ordinary analytic perturbation theory for a simple eigenvalue, applied to
$\Lambda+sB_0$ at $s=0$ \cite[Chapter~II]{Kato}, gives analytic eigenvalue and
spectral-projection branches.  Thus, for all sufficiently large positive
$t$, there is a unique real component $y_k^+(t)$ satisfying
\begin{equation}\label{eq:realbranch}
  y_k^+(t)=(c_k-c_m)t+\gamma_k+O(t^{-1}).
\end{equation}
This is the ``real eigenvalue branch'' used below.  We use only the standard
finite-dimensional expansion of a simple eigenvalue of $\Lambda+sB_0$; no
perturbation theory for the differential equation is involved.  For
sufficiently large $|t|$, the Rouch\'e argument from
Lemma~\ref{lem:localization} also isolates these branches, so no separate
branch ambiguity arises.

Let $F_k^+(t)$ be the spectral projection associated with $y_k^+(t)$ and
set $a_k^+(t)=\Tr(P_mF_k^+(t))$.  The first-order eigenvector expansion gives
\begin{align}
  a_m^+(t)
  &=1-4\sum_{k\ne m}\frac{1}{(c_k-c_m)^4t^2}
    +O(t^{-3}),\label{eq:weightm}\\
  a_k^+(t)
  &=\frac{4}{(c_k-c_m)^4t^2}+O(t^{-3}),
  \qquad k\ne m.\label{eq:weightk}
\end{align}
Indeed, a unit eigenvector $u_k^+(t)$ with eigenvalue $y_k^+(t)$ satisfies,
for $k\ne m$,
\[
  \ip{u_k^+(t)}{e_m}
  =\frac{1}{t}\,
  \frac{(B_0)_{mk}}{c_k-c_m}
  +O(t^{-2})
  =
  -\frac{2\ii}{(c_k-c_m)^2t}+O(t^{-2}),
\]
which proves~\eqref{eq:weightk}.  The identity
$\sum_{k=1}^Na_k^+(t)=1$ then gives~\eqref{eq:weightm}.

\begin{proposition}[Fixed-$\eps$ pole-height asymptotics]
\label{prop:fixedheight}
Fix $\eps>0$.  For all sufficiently large positive $t$, label the eigenvalue
$z_k^{\eps,+}(t)$ of $\mathcal M_\eps(t)$ by
\begin{equation}\label{eq:complexbranch}
  z_k^{\eps,+}(t)
  =(c_k-c_m)t+\gamma_k-\ii\eps\delta_{km}
  +O_\eps(t^{-1}).
\end{equation}
Then
\begin{align}
  -\Im z_m^{\eps,+}(t)
  &=
  \eps
  -4\eps\sum_{k\ne m}
    \frac{1}{(c_k-c_m)^4t^2}
  +O_\eps(t^{-3}),\label{eq:heightm}\\
  -\Im z_k^{\eps,+}(t)
  &=
  \frac{4\eps}{(c_k-c_m)^4t^2}
  +O_\eps(t^{-3}),
  \qquad k\ne m.\label{eq:heightk}
\end{align}
The same formulas hold as $t\to-\infty$ for the branches labelled by
\[
  z_k^{\eps,-}(t)
  =(c_k-c_m)t+\gamma_k-\ii\eps\delta_{km}
  +O_\eps(|t|^{-1}).
\]
\end{proposition}

\begin{proof}
Write
\begin{equation}\label{eq:Beps}
  \mathcal M_\eps(t)
  =t(\Lambda+sB_\eps),
  \qquad
  s=t^{-1},
  \qquad
  B_\eps=B_0-\ii\eps P_m.
\end{equation}
Let $\mu_k(s)$ be the analytic eigenvalue of $\Lambda+sB_\eps$ with
$\mu_k(0)=c_k-c_m$.  Its expansion is
\[
  \mu_k(s)
  =(c_k-c_m)
  +s(B_\eps)_{kk}
  +s^2\mu_{k,2}
  +s^3\mu_{k,3}
  +O_\eps(s^4).
\]
The second-order coefficient
\[
  \mu_{k,2}
  =\sum_{j\ne k}
  \frac{(B_\eps)_{kj}(B_\eps)_{jk}}{c_k-c_j}
\]
is real.  The third-order coefficient is
\begin{equation}\label{eq:thirdorder}
  \mu_{k,3}
  =
  \sum_{j,\ell\ne k}
  \frac{
    (B_\eps)_{kj}(B_\eps)_{j\ell}(B_\eps)_{\ell k}
  }{(c_k-c_j)(c_k-c_\ell)}
  -(B_\eps)_{kk}
  \sum_{j\ne k}
  \frac{(B_\eps)_{kj}(B_\eps)_{jk}}{(c_k-c_j)^2}.
\end{equation}

In the first sum in~\eqref{eq:thirdorder}, the imaginary parts of the terms
with $j\ne\ell$ cancel in pairs under $(j,\ell)\leftrightarrow(\ell,j)$.  When $j=\ell$, the term
is real unless $j=m$.  If $k\ne m$, the only imaginary contribution is
\[
  \Im
  \frac{
    (B_\eps)_{km}(B_\eps)_{mm}(B_\eps)_{mk}
  }{(c_k-c_m)^2}
  =
  -\frac{4\eps}{(c_k-c_m)^4}.
\]
For $k\ne m$, the normalization term in~\eqref{eq:thirdorder} is real
because $(B_\eps)_{kk}=\gamma_k$.

Thus, for each $k\ne m$, multiplication of the expansion for
$\mu_k(t^{-1})$ by $t$ gives~\eqref{eq:heightk}.
Since the branches $z_k^{\eps,+}(t)$, $k=1,\ldots,N$, exhaust the spectrum of
$\mathcal M_\eps(t)$ and
$-\Im\Tr\mathcal M_\eps(t)=\eps\Tr(P_m)=\eps$,
equation~\eqref{eq:heightm} follows:
\[
  -\Im z_m^{\eps,+}(t)
  =\eps-\sum_{k\ne m}\bigl(-\Im z_k^{\eps,+}(t)\bigr)
  =\eps-4\eps\sum_{k\ne m}\frac{1}{(c_k-c_m)^4t^2}
  +O_\eps(t^{-3}).
\]
For negative $t$, the same perturbative calculation applies with $s=t^{-1}<0$,
and the same exact trace identity gives the distinguished branch.  This proves
the stated formulas as $t\to-\infty$ as well.
\end{proof}

As a consistency check, the leading coefficients in
Proposition~\ref{prop:fixedheight} agree with
\eqref{eq:weightm}--\eqref{eq:weightk}.

\section{The two-particle example}\label{sec:example}

In this section we illustrate our approximation by providing an explicit
example in the case $N=2$.

Let $N=2$, write
\[
  c_1<c_2,
  \qquad
  g=c_2-c_1>0,
\]
and choose $m=1$.  We subtract $c_1t$ from both particle positions.  Define
\begin{equation}\label{eq:twoaux}
  T(t)=\gamma_1+\gamma_2+gt,
  \qquad
  \Delta(t)=\gamma_1-\gamma_2-gt,
  \qquad
  R(t)=\sqrt{\Delta(t)^2+\frac{16}{g^2}}.
\end{equation}
The two real components in this frame are
\begin{equation}\label{eq:tworeal}
  y_\pm(t)=\frac{T(t)\pm R(t)}{2},
  \qquad
  y_-(t)<y_+(t).
\end{equation}
The functions $y_\pm$ solve~\eqref{eq:realCM}, but their ordered labels are not
tied to fixed asymptotic velocity branches.  Indeed,
\begin{align*}
  y_+(t)&=gt+\gamma_2+O(t^{-1}),
  &y_-(t)&=\gamma_1+O(t^{-1})
  &&(t\to+\infty),\\
  y_+(t)&=\gamma_1+O(|t|^{-1}),
  &y_-(t)&=gt+\gamma_2+O(|t|^{-1})
  &&(t\to-\infty).
\end{align*}
This exchange of ordered labels clarifies the roles of $\gamma_1$ and
$\gamma_2$.

The corresponding approximating multisoliton matrix is
\begin{equation}\label{eq:twomatrix}
  \mathcal M_\eps(t)
  =
  \begin{pmatrix}
    \gamma_1-\ii\eps & -2\ii/g\\
    2\ii/g & \gamma_2+gt
  \end{pmatrix}.
\end{equation}
Consequently, the inverse formula~\eqref{eq:GLinverse} becomes the explicit
two-soliton profile
\begin{equation}\label{eq:twoprofile}
  q^\eps(t,x)
  =
  \sqrt{2\eps}\,\e^{\ii\phi}
  \frac{\gamma_2+gt-x+2\ii/g}
  {(\gamma_1-\ii\eps-x)(\gamma_2+gt-x)-4/g^2}.
\end{equation}
Its poles are
\begin{equation}\label{eq:twopoles}
  z_\pm^\eps(t)
  =
  \frac{T(t)-\ii\eps\pm R_\eps(t)}{2},
  \qquad
  R_\eps(t)
  :=
  \sqrt{(\Delta(t)-\ii\eps)^2+\frac{16}{g^2}},
\end{equation}
where the square-root branch is continued from $R_0(t)=R(t)>0$.

The spectral weights are
\begin{equation}\label{eq:twoweights}
  a_+(t)
  =\frac12\left(1+\frac{\Delta(t)}{R(t)}\right),
  \qquad
  a_-(t)
  =\frac12\left(1-\frac{\Delta(t)}{R(t)}\right).
\end{equation}
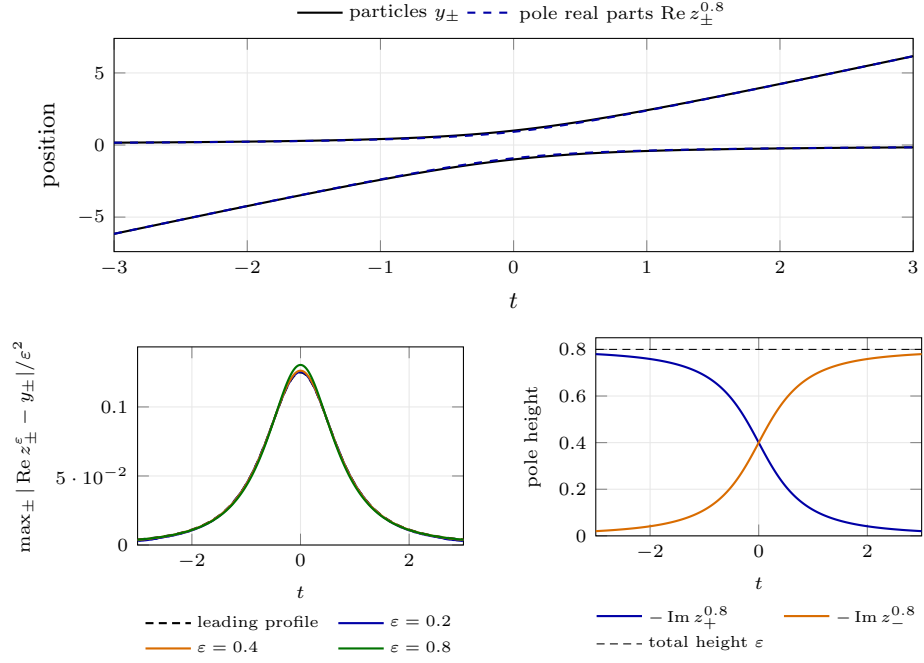
\begin{figure}[t]
\centering
\begin{tikzpicture}
  \begin{axis}[
    width=0.96\textwidth,
    height=4.4cm,
    domain=-3:3,
    samples=241,
    xlabel={$t$},
    ylabel={position},
    xmin=-3,
    xmax=3,
    grid=major,
    grid style={gray!18},
    tick label style={font=\scriptsize},
    label style={font=\small},
    legend style={
      font=\scriptsize,
      draw=none,
      fill=none,
      at={(0.5,1.02)},
      anchor=south,
      legend columns=2
    }
  ]
    \addplot[black,thick] {x+sqrt(x^2+1)};
    \addlegendentry{particles $y_\pm$}
    \addplot[black,thick,forget plot] {x-sqrt(x^2+1)};
    \addplot[blue!65!black,dashed,thick] {x+0.5*twou(x,0.8)};
    \addlegendentry{pole real parts $\Re z_\pm^{0.8}$}
    \addplot[blue!65!black,dashed,thick,forget plot]
      {x-0.5*twou(x,0.8)};
  \end{axis}
\end{tikzpicture}

\vspace{0.4em}

\begin{minipage}{0.49\textwidth}
\centering
\begin{tikzpicture}
  \begin{axis}[
    width=0.95\textwidth,
    height=4.2cm,
    domain=-3:3,
    samples=241,
    xlabel={$t$},
    ylabel={$\max_\pm|\Re z_\pm^\eps-y_\pm|/\eps^2$},
    xmin=-3,
    xmax=3,
    ymin=0,
    grid=major,
    grid style={gray!18},
    tick label style={font=\scriptsize},
    label style={font=\scriptsize},
    legend style={
      font=\tiny,
      draw=none,
      fill=none,
      at={(0.5,-0.30)},
      anchor=north,
      legend columns=2,
      /tikz/every even column/.append style={column sep=0.5em}
    },
    legend cell align=left
  ]
    \addplot[black,densely dashed,thick]
      {1/(8*(x^2+1)^(3/2))};
    \addlegendentry{leading profile}
    \addplot[blue!65!black,thick]
      {abs(twou(x,0.2)-2*sqrt(x^2+1))/(2*0.2^2)};
    \addlegendentry{$\eps=0.2$}
    \addplot[orange!85!black,thick]
      {abs(twou(x,0.4)-2*sqrt(x^2+1))/(2*0.4^2)};
    \addlegendentry{$\eps=0.4$}
    \addplot[green!45!black,thick]
      {abs(twou(x,0.8)-2*sqrt(x^2+1))/(2*0.8^2)};
    \addlegendentry{$\eps=0.8$}
  \end{axis}
\end{tikzpicture}
\end{minipage}
\hfill
\begin{minipage}{0.49\textwidth}
\centering
\begin{tikzpicture}
  \begin{axis}[
    width=0.95\textwidth,
    height=4.2cm,
    domain=-3:3,
    samples=241,
    xlabel={$t$},
    ylabel={pole height},
    xmin=-3,
    xmax=3,
    ymin=0,
    ymax=0.85,
    grid=major,
    grid style={gray!18},
    tick label style={font=\scriptsize},
    label style={font=\scriptsize},
    legend style={
      font=\tiny,
      draw=none,
      fill=none,
      at={(0.5,-0.30)},
      anchor=north,
      legend columns=2,
      /tikz/every even column/.append style={column sep=0.5em}
    },
    legend cell align=left
  ]
    \addplot[blue!65!black,thick] {(0.8-twov(x,0.8))/2};
    \addlegendentry{$-\Im z_+^{0.8}$}
    \addplot[orange!85!black,thick] {(0.8+twov(x,0.8))/2};
    \addlegendentry{$-\Im z_-^{0.8}$}
    \addplot[black,densely dashed] {0.8};
    \addlegendentry{total height $\eps$}
  \end{axis}
\end{tikzpicture}
\end{minipage}
\caption{Numerical illustration of the exact two-particle formulas with
$g=2$ and $\gamma_1=\gamma_2=0$.  Top: the pole real parts track the real
particles through the interaction.  Bottom left: after division by
$\eps^2$, the real-position errors approach the leading profile
$[8(1+t^2)^{3/2}]^{-1}$.  Bottom right: the two pole heights exchange while
their sum remains $\eps$.}
\label{fig:two-particle}
\end{figure}

Expanding~\eqref{eq:twopoles} gives
\begin{align}
  z_+^\eps(t)
  &=
  y_+(t)-\ii\eps a_+(t)
  -\eps^2\frac{a_+(t)a_-(t)}{R(t)}
  +O(\eps^3),\label{eq:twoplusexp}\\
  z_-^\eps(t)
  &=
  y_-(t)-\ii\eps a_-(t)
  +\eps^2\frac{a_+(t)a_-(t)}{R(t)}
  +O(\eps^3),\label{eq:twominusexp}
\end{align}
uniformly for $t\in\R$.  These formulas show directly that the displacements
$z_\pm^\eps-y_\pm$ are not purely imaginary: their real parts are generally
nonzero at order $\eps^2$.

Since
\[
  \frac{a_+(t)a_-(t)}{R(t)}
  =\frac{4}{g^2R(t)^3},
\]
the leading real error is largest when $\Delta(t)=0$, that is, at the time of
closest interaction.  For each real $\Delta$, the branch points of
$R_\eps$ as a function of complex $\eps$ have modulus
\[
  \sqrt{\Delta^2+\frac{16}{g^2}}\ge\frac4g.
\]
Hence the real part of~\eqref{eq:twopoles} has a uniform even expansion in
$\eps$ for $|\eps|<4/g$.  The leading coefficient has its unique
nondegenerate maximum at $\Delta=0$, so perturbing this maximum changes its
value only at the next even order.  Since $R=4/g$ at $\Delta=0$, one obtains
\begin{equation}\label{eq:twomax}
  \sup_{t\in\R}
  |\Re z_\pm^\eps(t)-y_\pm(t)|
  =
  \frac{g}{16}\eps^2+O(\eps^4).
\end{equation}
Thus the $O(\eps^2)$ real-position estimate in
Theorem~\ref{thm:quantitative} is sharp.

\FloatBarrier
\section*{Statements and Declarations}

\ifx\FundingStatement\empty\else
\noindent\textbf{Funding.} \FundingStatement\par\medskip
\fi
\ifx\CompetingInterestsStatement\empty\else
\noindent\textbf{Competing interests.} \CompetingInterestsStatement\par\medskip
\fi
\noindent\textbf{Data availability.} All mathematical results are proved in
the article. The numerical illustration in Figure~\ref{fig:two-particle} is
obtained directly from the explicit formulas in Section~\ref{sec:example};
the plotting code is included in the LaTeX source.

\medskip
\noindent\textbf{Use of generative AI.} Generative AI tools assisted with
organization and language editing. The author verified all mathematical
arguments and references and is solely responsible for the manuscript.


\begingroup
\footnotesize
\begin{thebibliography}{99}
\setlength{\itemsep}{0pt}
\setlength{\parsep}{0pt}

\bibitem{AbanovBettelheimWiegmann}
A.~G. Abanov, E.~Bettelheim, and P.~Wiegmann,
\emph{Integrable hydrodynamics of Calogero--Sutherland model:
bidirectional Benjamin--Ono equation},
J. Phys. A \textbf{42} (2009), no.~13, 135201, 24 pp.
\url{https://doi.org/10.1088/1751-8113/42/13/135201}.

\bibitem{BadreddineGlobal}
R.~Badreddine,
\emph{On the global well-posedness of the Calogero--Sutherland derivative
nonlinear Schr\"odinger equation},
Pure Appl. Anal. \textbf{6} (2024), no.~2, 379--414.
\url{https://doi.org/10.2140/paa.2024.6.379}.

\bibitem{BadreddineZeroDisp}
R.~Badreddine,
\emph{Zero dispersion limit of the Calogero--Moser derivative NLS equation},
SIAM J. Math. Anal. \textbf{56} (2024), no.~6, 7228--7249.
\url{https://doi.org/10.1137/24M1646935}.

\bibitem{CalogeroGround}
F.~Calogero,
\emph{Ground state of a one-dimensional $n$-body system},
J. Math. Phys. \textbf{10} (1969), no.~12, 2197--2200.
\url{https://doi.org/10.1063/1.1664821}.

\bibitem{Calogero}
F.~Calogero,
\emph{Solution of the one-dimensional $N$-body problems with quadratic
and/or inversely quadratic pair potentials},
J. Math. Phys. \textbf{12} (1971), 419--436.
\url{https://doi.org/10.1063/1.1665604}.

\bibitem{ChapoutoForlanoLaurens}
A.~Chapouto, J.~Forlano, and T.~Laurens,
\emph{On the well-posedness of the intermediate nonlinear Schr\"odinger
equation on the line},
preprint, \href{https://arxiv.org/abs/2511.00302}{arXiv:2511.00302}, 2025.

\bibitem{ChapoutoForlanoLaurensSubcritical}
A.~Chapouto, J.~Forlano, and T.~Laurens,
\emph{Sub-critical well-posedness for the intermediate nonlinear Schr\"odinger
equation on the line},
preprint, \href{https://arxiv.org/abs/2608.06298}{arXiv:2608.06298}, 2026.

\bibitem{FrankRead}
R.~L. Frank and L.~Read,
\emph{Jost solutions and direct scattering for the continuum
Calogero--Moser equation},
preprint, \href{https://arxiv.org/abs/2510.11403}{arXiv:2510.11403}, 2025.

\bibitem{GerardLenzmann}
P.~G\'erard and E.~Lenzmann,
\emph{The Calogero--Moser derivative nonlinear Schr\"odinger equation},
Comm. Pure Appl. Math. \textbf{77} (2024), 4008--4062.
\url{https://doi.org/10.1002/cpa.22203}.

\bibitem{Hadama}
S.~Hadama,
\emph{Small-data $L^2$ theory for the intermediate NLS and the
Calogero--Moser derivative NLS},
preprint, \href{https://arxiv.org/abs/2608.01138}{arXiv:2608.01138}, 2026.

\bibitem{HoganKowalski}
J.~Hogan and M.~Kowalski,
\emph{Turbulent threshold for continuum Calogero--Moser models},
Pure Appl. Anal. \textbf{6} (2024), 941--954.
\url{https://doi.org/10.2140/paa.2024.6.941}.

\bibitem{Kato}
T.~Kato,
\emph{Perturbation Theory for Linear Operators},
Classics in Mathematics, Springer, Berlin, 1995.
\url{https://doi.org/10.1007/978-3-642-66282-9}.

\bibitem{KillipLaurensVisan}
R.~Killip, T.~Laurens, and M.~Vi\c{s}an,
\emph{Scaling-critical well-posedness for continuum Calogero--Moser models
on the line},
Commun. Amer. Math. Soc. \textbf{5} (2025), 284--320.
\url{https://doi.org/10.1090/cams/48}.

\bibitem{KillipMarsdenVisan}
R.~Killip, K.~Marsden, and M.~Vi\c{s}an,
\emph{The Hamiltonian formulation of continuum Calogero--Moser models},
preprint, \href{https://arxiv.org/abs/2604.09479}{arXiv:2604.09479}, 2026.

\bibitem{KimKwon}
T.~Kim and S.~Kwon,
\emph{Soliton resolution for Calogero--Moser derivative nonlinear
Schr\"odinger equation},
to appear in J. Eur. Math. Soc., \href{https://arxiv.org/abs/2408.12843}{arXiv:2408.12843}, 2024.

\bibitem{Matsuno}
Y.~Matsuno,
\emph{Multiphase solutions and their reductions for a nonlocal nonlinear
Schr\"odinger equation with focusing nonlinearity},
Stud. Appl. Math. \textbf{151} (2023), 883--922.
\url{https://doi.org/10.1111/sapm.12610}.

\bibitem{Moser}
J.~Moser,
\emph{Three integrable Hamiltonian systems connected with isospectral
deformations},
Adv. Math. \textbf{16} (1975), 197--220.
\url{https://doi.org/10.1016/0001-8708(75)90151-6}.

\bibitem{Pelinovsky}
D.~E. Pelinovsky,
\emph{Intermediate nonlinear Schr\"odinger equation for internal waves in a
fluid of finite depth},
Phys. Lett. A \textbf{197} (1995), nos.~5--6, 401--406.
\url{https://doi.org/10.1016/0375-9601(94)00991-W}.

\bibitem{Ruijsenaars}
S.~N.~M. Ruijsenaars,
\emph{Action-angle maps and scattering theory for some finite-dimensional
integrable systems. I. The pure soliton case},
Comm. Math. Phys. \textbf{115} (1988), 127--165.
\url{https://doi.org/10.1007/BF01238855}.

\bibitem{Sutherland71}
B.~Sutherland,
\emph{Exact results for a quantum many-body problem in one dimension},
Phys. Rev. A \textbf{4} (1971), 2019--2021.
\url{https://doi.org/10.1103/PhysRevA.4.2019}.

\bibitem{Sutherland72}
B.~Sutherland,
\emph{Exact results for a quantum many-body problem in one dimension. II},
Phys. Rev. A \textbf{5} (1972), 1372--1376.
\url{https://doi.org/10.1103/PhysRevA.5.1372}.

\bibitem{Sutherland75}
B.~Sutherland,
\emph{Exact ground-state wave function for a one-dimensional plasma},
Phys. Rev. Lett. \textbf{34} (1975), 1083--1085.
\url{https://doi.org/10.1103/PhysRevLett.34.1083}.

\bibitem{SunWangZhao}
R.~Sun, D.-S.~Wang, and Y.~Zhao,
\emph{On the direct scattering theory for the Calogero--Moser derivative
nonlinear Schr\"odinger equation},
preprint, \href{https://arxiv.org/abs/2607.11123}{arXiv:2607.11123}, 2026.

\end{thebibliography}
\endgroup
\end{document}